\documentclass[a4paper,11pt]{amsart}

\usepackage{amsfonts,amssymb}
\usepackage{graphicx,tikz}
\usepackage{float}
\usepackage{hyperref}

\theoremstyle{plain}

\newtheorem{thm}{Theorem}[section]
\newtheorem{lem}[thm]{Lemma}

\newtheorem*{thmA}{Theorem A}
\newtheorem*{question}{Question}
\theoremstyle{definition}

\newtheorem{dfn}[thm]{Definition}
\newtheorem{rmk}[thm]{Remark}

\begin{document}

\title
{A focal subgroup theorem for outer commutator words}

\author[C. Acciarri]{Cristina Acciarri}
\address{Department of Mathematics\\ University of Brasilia
\\ Brasilia-DF\\ 70910-900 Brazil}
\email{acciarricristina@yahoo.it}

\author[G.A. Fern\'andez-Alcober]{Gustavo A. Fern\'andez-Alcober}
\address{Matematika Saila\\ Euskal Herriko Unibertsitatea
\\ 48080 Bilbao\\ Spain}
\email{gustavo.fernandez@ehu.es}

\author[P. Shumyatsky]{Pavel Shumyatsky}
\address{Department of Mathematics\\ University of Brasilia
\\ Brasilia-DF\\ 70910-900 Brazil}
\email{pavel@unb.br}

\thanks{The second author is supported by the Spanish Government, grant
MTM2008-06680-C02-02, partly with FEDER funds, and by the
Basque Government, grant IT-460-10.
The third author is supported by CNPq-Brazil.}

\begin{abstract}
Let $G$ be a finite group of order $p^am$, where $p$ is a prime and
$m$ is not divisible by $p$, and let $P$ be a Sylow $p$-subgroup of $G$.
If $w$ is an outer commutator word, we prove that $P\cap w(G)$ is
generated by the intersection of $P$ with the set of $m$th powers
of all values of $w$ in $G$.
\end{abstract}

\maketitle

Let $G$ be a finite group  and $P$  a Sylow $p$-subgroup of $G$.
The Focal Subgroup Theorem states that $P\cap G'$ is generated by the set of commutators $\{ [g,z] \mid g\in G,\ z\in P,\ [g,z]\in P \}$. This was proved by Higman in 1953 \cite{higman}. Nowadays the proof of the theorem can be found in many standard books on group theory (for example, Rose's book \cite{rose} or Gorenstein's \cite{gore}).

One immediate corollary is that $P\cap G'$ can be generated by commutators  lying in $P$. Of course, $G'$ is the verbal subgroup of $G$ corresponding to the group word $[x,y]=x^{-1}y^{-1}xy$. It is natural to ask the question on generation of Sylow subgroups  for other words. More specifically, if $w$ is a group word  we write $G_w$ for the set of values of $w$ in $G$ and $w(G)$ for the subgroup generated by $G_w$ (which is called the \emph{verbal subgroup\/} of $w$ in $G$), and  one is tempted to ask the following question.
\begin{question}
Given  a Sylow $p$-subgroup $P$ of a finite group $G$, is it true that
$P\cap w(G)$ can be generated by $w$-values lying in $P$, i.e., that $P\cap w(G)=\langle P\cap G_w \rangle$?
\end{question}
However considering the case where $G$ is the non-abelian group of order six,  $w=x^3$ and $p=3$ we quickly see that the answer to the above question is negative. Therefore we concentrate on the case where $w$ is a \emph{commutator} word. Recall that a group word is commutator if the sum of the exponents of any indeterminate involved in it is zero. Thus, we deal with the question whether $P\cap w(G)$ can be generated by $w$-values whenever $w$ is a commutator word.

The main result of this paper is a contribution towards a positive answer to this question: we prove that if $w$ is an outer commutator word, then $P\cap w(G)$ can be generated by the \emph{powers of values of $w$\/} which lie in $P$.
More precisely, we have the following result.

\begin{thmA}
Let $G$ be a finite group of order $p^am$, where $p$ is a prime and
$m$ is not divisible by $p$, and let $P$ be a Sylow $p$-subgroup of $G$.
If $w$ is an outer commutator word, then $P\cap w(G)$ is generated by $m$th powers of $w$-values, i.e., $P\cap w(G)=\langle P\cap G_{w^m} \rangle$.
\end{thmA}

Recall that an \emph{outer commutator word\/} is a word which is obtained by nesting commutators,
but using always \emph{different indeterminates\/}.
Thus the word \break $[[x_1,x_2],[x_3,x_4,x_5],x_6]$ is an outer commutator while the Engel word \break
$[x_1,x_2,x_2,x_2]$ is not.
An important family of outer commutator words are the simple commutators $\gamma_i$,
given by
\[
\gamma_1=x_1,
\qquad
\gamma_i=[\gamma_{i-1},x_i]=[x_1,\ldots,x_i],
\quad
\text{for $i\ge 2$.}
\]
The corresponding verbal subgroups $\gamma_i(G)$ are the terms of the lower central series of $G$. Another distinguished sequence of outer commutator words are the \emph{derived words\/} $\delta_i$, on $2^i$ indeterminates, which are defined recursively by
\[
\delta_0=x_1,
\qquad
\delta_i=[\delta_{i-1}(x_1,\ldots,x_{2^{i-1}}),\delta_{i-1}(x_{2^{i-1}+1},\ldots,x_{2^i})],
\quad
\text{for $i\ge 1$.}
\]
Then $\delta_i(G)=G^{(i)}$, the $i$th derived subgroup of $G$.

Some of the ideas behind the proof of Theorem A were anticipated already in \cite{gushu} where somewhat similar arguments, due to Guralnick, led to a result on generation of a Sylow $p$-subgroup of $G'$ for a finite group $G$ admitting a coprime group of automorphisms. Later the arguments were refined in \cite{accshu}. In both papers \cite{gushu} and \cite{accshu} the results on generation of Sylow subgroups were used to reduce a problem about finite groups to the case of nilpotent groups. It is hoped that also our Theorem A will play a similar role in the subsequent projects.

Another important tool used in the proof of Theorem A  is the famous result of  Liebeck, O'Brien, Shalev and Tiep \cite{LOST} that every element of a non-abelian simple group is a commutator. The result proved Ore's conjecture thus solving a long-standing problem. In turn, the proof in \cite{LOST} uses the classification of finite simple groups as well as many other sophisticated tools.

\section{Preliminaries}

If $X$ and $Y$ are two subsets of a group $G$, and $N$ is a normal
subgroup of $G$, it is not always the case that $XN\cap YN=(X\cap Y)N$,
i.e.,\ that $\overline X\cap \overline Y=\overline{X\cap Y}$ in the
quotient group $\overline G=G/N$.
In our first lemma we have a situation in which this property holds, and
which will be of importance in the sequel.

\begin{lem}
\label{intersection mod N}
Let $G$ be a finite group, and let $N$ be a normal subgroup of $G$.
If $P$ is a Sylow $p$-subgroup of $G$ and $X$ is a normal subset of $G$
consisting of $p$-elements, then $XN\cap PN=(X\cap P)N$.
In other words, if we use the bar notation in $G/N$, we have
$\overline{X}\cap \overline{P}=\overline{X\cap P}$.
\end{lem}

\begin{proof}
We only need to worry about the inclusion $\overline{X}\cap \overline{P}\subseteq \overline{X\cap P}$.
So, given an element $g\in XN\cap PN$, we prove that $g\in xN$ for some
$x\in X\cap P$.
Since we have $g\in XN$, we may assume without loss of generality that $g\in X$,
and in particular $g$ is a $p$-element.
Since also $g\in PN$, there exists $z\in P$ such that $gN=zN$.

Put $H=\langle g \rangle N=\langle z \rangle N$, and observe that
$H'\le N$.
Since $P\cap N$ is a Sylow $p$-subgroup of $N$ and $z\in P$, it
follows that $P\cap H=\langle z \rangle (P\cap N)$ is a Sylow $p$-subgroup
of $H$.
Now, $g$ is a $p$-element of $H$, and so we get $g^h\in P\cap H$ for some $h\in H$.
If we put $x=g^h$ then $x\in X\cap P$, since $X$ is a normal subset of $G$,
and $g=x^{h^{-1}}=x[x,h^{-1}]\in xH'\subseteq xN$, as desired.
\end{proof}

The next lemma will be fundamental in the proof of Theorem A, since it
will allow us to go up a series from $1$ to $w(G)$ in which all quotients of
two consecutive terms are verbal subgroups of a word all of whose values are
also $w$-values.

\begin{lem}
\label{from N to L}
Let $G$ be a finite group, and let $P$ be a Sylow $p$-subgroup of $G$.
Assume that $N\le L$ are two normal subgroups of $G$, and use the
bar notation in the quotient group $G/N$. Let $X$ be a normal subset of $G$ consisting of $p$-elements such that $\overline P\cap \overline L=\langle \overline P\cap \overline X \rangle$. Then $P\cap L=\langle P\cap X, P\cap N \rangle$.
\end{lem}

\begin{proof}
By Lemma \ref{intersection mod N}, we have
$\overline P \cap \overline L=\langle \overline{P\cap X} \rangle$, and this
implies that $PN\cap L=\langle P\cap X \rangle N$.
By intersecting with $P$, we get
\[
P\cap L = P\cap (PN\cap L)
= P\cap \langle P\cap X \rangle N
= \langle P\cap X \rangle (P\cap N),
\]
where the last equality follows from Dedekind's law.
This proves the result.
\end{proof}

We will also need the following lemma, of a different nature.

\begin{lem}
\label{no deltai values}
Let $G$ be a finite group, and let $N$ be a minimal normal subgroup of $G$.
If $N$ does not contain any non-trivial elements of $G_{\delta_i}$, where $i\ge 1$,
then $[N,G^{(i-1)}]=1$.
\end{lem}

\begin{proof}
We argue by induction on $i$.
If $i=1$ then, since $N$ is normal in $G$ and does not contain any non-trivial commutators
of elements of $G$, it follows that $[n,g]=1$ for every $n\in N$ and $g\in G$.
Thus $[N,G]=1$, as desired.

Assume now that $i\ge 2$.
The fact that $N$ is a minimal normal subgroup of $G$ implies that the subgroup $\langle N\cap G_{\delta_{i-1}}\rangle$ must be either equal to $N$ or the trivial subgroup. In the former case, we have $N=\langle N\cap G_{\delta_{i-1}}\rangle$ and so $[N,G^{(i-1)}]$ is generated by elements of the form $[a,b]$ where $a\in N\cap G_{\delta_{i-1}}$ and $b\in G_{\delta_{i-1}}$.
In particular, each commutator $[a,b]$ belongs to $N\cap G_{\delta_{i}}$ and must be $1$ by the hypothesis. Hence $[N,G^{(i-1)}]=1$. If  $N\cap G_{\delta_{i-1}}=1$, then it follows from the induction hypothesis that $[N,G^{(i-2)}]=1$, and the result holds.
\end{proof}

We conclude this preliminary section by showing that Theorem A holds
\emph{for every word\/} under the assumption that the verbal subgroup
$w(G)$ is nilpotent.

\begin{thm}
\label{nilpotent case}
Let $G$ be a finite group of order $p^am$, where $p$ is a prime and
$m$ is not divisible by $p$, and let $P$ be a Sylow $p$-subgroup of $G$.
If $w$ is any word such that $w(G)$ is nilpotent, then
\[
P\cap w(G)=\langle P\cap G_{w^m} \rangle.
\]
\end{thm}

\begin{proof}
By Bezout's identity, there exist two integers $\lambda$ and $\mu$ such that
$1=\lambda p^a+\mu m$.
If we put $u=w^{p^a}$ and $v=w^m$, then for every $g\in G_w$ we have
\[
g=(g^{p^a})^{\lambda} \, \cdot(g^m)^{\mu} \in \langle G_u\rangle \cdot \langle G_v\rangle.
\]
Hence
\begin{equation}
\label{other gens}
w(G) = \langle G_u, G_v \rangle.
\end{equation}
Note that all elements of $G_u$ have $p'$-order, and all elements of $G_v$
have $p$-power order.
Since $w(G)$ is nilpotent, it follows that $\langle G_u \rangle$ is a $p'$-subgroup
of $w(G)$, $\langle G_v \rangle$ is a $p$-subgroup, and $G_{u}$ and $G_{v}$ commute
elementwise.
As a consequence of this and (\ref{other gens}), we get
\begin{equation}
\label{change_gen_set}
w(G) = \langle G_u \rangle \times \langle G_v\rangle,
\end{equation}
and $\langle G_u \rangle$ and $\langle G_v \rangle$ are a Hall $p'$-subgroup
and a Sylow $p$-subgroup of $w(G)$, respectively.
We conclude that $P\cap w(G)=\langle G_v \rangle$, which proves the theorem. 
\end{proof}

\section{The proof of Theorem A}

The first step in the proof of Theorem A is to verify it for $\delta_i$, which is done
in Theorem \ref{deltai} below.
For this, we will rely on the result by Liebeck, O'Brien, Shalev and
Tiep \cite{LOST} that proved  Ore's conjecture, according to which every element of a non-abelian
simple group is a commutator, and \textit{a fortiori\/}, also a value of $\delta_i$
for every $i$.
We will also need the following result of Gasch\"utz (see page 191 of \cite{gaschutz}).

\begin{thm}
\label{gaschutz}
Let $G$ be a finite group, and let $P$ be a Sylow $p$-subgroup of $G$.
If $N$ is a normal abelian p-subgroup of $G$, then $N$ is complemented in $G$ if
and only if $N$ is complemented in $P$.
\end{thm}

In the proof of Theorem A for both $\delta_i$ and an arbitrary outer commutator word,
we will argue by induction.
Then it will be important to take into account the following remark.

\begin{rmk}
Let $G$ be a group of order $p^am$ for which we want to prove Theorem A
in the case of a given word $w$.
Assume that $K$ is a group whose order $p^{a^*}m^*$ is a divisor of $p^am$ (for example,
a subgroup or a quotient of $G$), and let $P^*$ be a Sylow $p$-subgroup of $K$.
If Theorem A is known to hold for $K$ and $w$, then we have
$P^*\cap w(K)=\langle P^*\cap K_{w^{m^*}} \rangle$.
Since $m/m^*$ is a positive integer which is coprime to $p$, it follows that
$P^*\cap w(K)=\langle (P^*\cap K_{w^{m^*}})^{m/m^*} \rangle$, and so also that
$P^*\cap w(K)=\langle P^*\cap K_{w^m} \rangle$.
In other words, in the statement of Theorem A for $K$, we can replace the power
word $w^{m^*}$ corresponding to the order of $K$ with the word $w^m$, which
corresponds to the order of $G$.
\end{rmk}

We can now proceed to the proof of Theorem A for $\delta_i$.

\begin{thm}
\label{deltai}
Let $G$ be a finite group of order $p^am$, where $p$ is a prime and
$m$ is not divisible by $p$, and let $P$ be a Sylow $p$-subgroup of $G$.
Then, for every $i\ge 0$, we have
\[
P\cap G^{(i)}=\langle P\cap G_{\delta_i^m} \rangle.
\]
\end{thm}

\begin{proof}
We argue by induction on the order of $G$.
The result is obvious if either $i=0$ or $G^{(i)}=1$, so we assume
that $i\ge 1$ and $G^{(i)}\ne 1$.

Let $N\ne 1$ be a normal subgroup of $G$ which is contained in
$G^{(i)}$.
Then the result holds in $\overline G=G/N$, and we have
$\overline P\cap \overline G^{(i)}
=\langle \overline P\cap \overline G_{\delta_i^m} \rangle$.
By applying Lemma \ref{from N to L}, we get
\begin{equation}
\label{P intersect Gi}
P\cap G^{(i)} = \langle P\cap G_{\delta_i^m}, P\cap N \rangle.
\end{equation}
Now we assume that $N$ is a minimal normal subgroup of $G$, and we consider
three different cases, depending on the structure of $N$.

\vspace{3pt}

\noindent
(i)
$N$ is a direct product of isomorphic non-abelian simple groups.

\vspace{3pt}

By the positive solution to Ore's conjecture, we have $N=N_{\delta_i}$.
Hence we get $P\cap N\subseteq N_{\delta_i}$, and since the map $z\mapsto z^m$
is a bijection in $P\cap N$, it follows that
$P\cap N\subseteq P\cap N_{\delta_i^m}$.
Now the result is immediate from (\ref{P intersect Gi}).

\vspace{3pt}

\noindent
(ii)
$N\cong C_q\times \cdots \times C_q$, where $q$ is a prime different from $p$.

\vspace{3pt}

In this case, $P\cap N=1$ and the result obviously holds.

\vspace{3pt}

\noindent
(iii)
$N\cong C_p\times \cdots \times C_p$.

\vspace{3pt}

In this case, we have $N\le P$ and so $P\cap N=N$.
Since $\langle N\cap G_{\delta_i} \rangle$ is a normal subgroup of $G$ and $N$
is a minimal normal subgroup, it follows that  either
$\langle N\cap G_{\delta_i} \rangle=N$ or $N\cap G_{\delta_i}=1$.
In the former case, we have $N=\langle (N\cap G_{\delta_i})^m \rangle$,
since $N$ is a finite $p$-group, and so $N=\langle N\cap G_{\delta_i^m} \rangle$
and the theorem follows again from (\ref{P intersect Gi}).
So we are necessarily in the latter case, and then by Lemma \ref{no deltai values},
we have $[N,G^{(i-1)}]=1$.

If $G$ is not perfect, then the theorem holds by induction in $G'$, and so
$P\cap G^{(i+1)}=P\cap (G')^{(i)}$ can be generated by values of $\delta_i^m$
lying in $P$.
If $G^{(i+1)}\ne 1$ then we can use (\ref{P intersect Gi}) with $G^{(i+1)}$
in the place of $N$, and we are done.
On the other hand, if $G^{(i+1)}=1$ then $G^{(i)}$ is abelian, and the result
is a consequence of Theorem \ref{nilpotent case}.

Thus we may assume that $G$ is perfect.
Then $P\cap G^{(i)}=P$.
Also $[N,G]=[N,G^{(i-1)}]=1$, and $N$ is central in $G$.
Being a minimal normal subgroup of $G$, this implies that $|N|=p$.
If $N\le \Phi(P)$ then it follows from (\ref{P intersect Gi}) that
$P=\langle P\cap G_{\delta_i^m} \rangle$, as desired.
Hence we may assume that $N$ is not contained in a maximal subgroup $M$
of $P$.
Since $|N|=p$, it follows that $M$ is a complement of $N$ in $P$.
By Theorem \ref{gaschutz}, it follows that $N$ has also a complement in $G$,
say $H$.
Since $N\le Z(G)$, we conclude that $G=H\times N$, a contradiction with the fact
that $G$ is perfect.
This completes the proof.
\end{proof}

\vspace{-0,01cm}

We will deal with arbitrary outer commutator words using some concepts from
the paper \cite{fernandez-morigi}, where outer commutator words are represented  by binary rooted trees in the following way: indeterminates are represented by an isolated vertex, and if $w=[u,v]$ is the commutator of two outer commutator words $u$ and $v$, then the tree $T_w$ of $w$ is obtained by drawing the trees $T_u$ and $T_v$, and a new vertex (which will be the root of the new tree)
which is then connected to the roots of $T_u$ and $T_v$.
For example, the following are the trees for the words $\gamma_4$ and $\delta_3$
(we also label every vertex with the outer commutator word which is represented by the tree appearing on top of that vertex):
\begin{figure}[H]
\centering
\begin{tikzpicture}[level distance=5mm]
\tikzstyle{level 1}=[sibling distance=10mm]
\tikzstyle{level 2}=[sibling distance=10mm]
\tikzstyle{level 3}=[sibling distance=10mm]
\coordinate (root)[grow=up, fill] circle (2pt)
child {[fill] circle (2pt)}
child {[fill] circle (2pt)
       child {[fill] circle (2pt)}
       child {[fill] circle (2pt)
               child {[fill] circle (2pt)}
               child {[fill] circle (2pt)}
             }
      };
\node[below=2pt] at (root) {\scriptsize $\gamma_4$};
\node[left=4pt] at (root-2) {\scriptsize $[x_1,x_2,x_3]$};
\node[right=2pt] at (root-1) {\scriptsize $x_4$};
\node[left=2pt] at (root-2-2) {\scriptsize $[x_1,x_2]$};
\node[right=2pt] at (root-2-1) {\scriptsize $x_3$};
\node[above=2pt] at (root-2-2-2) {\scriptsize $x_1$};
\node[above=2pt] at (root-2-2-1) {\scriptsize $x_2$};
\end{tikzpicture}
\qquad
\begin{tikzpicture}[level distance=5mm]
\tikzstyle{level 1}=[sibling distance=26mm]
\tikzstyle{level 2}=[sibling distance=15mm]
\tikzstyle{level 3}=[sibling distance=5mm]
\coordinate (root)[grow=up, fill] circle (2pt)
child {[fill] circle (2pt)
       child {[fill] circle (2pt)
               child {[fill] circle (2pt)}
               child {[fill] circle (2pt)}
             }
       child {[fill] circle (2pt)
               child {[fill] circle (2pt)}
               child {[fill] circle (2pt)}
             }
      }
child {[fill] circle (2pt)
       child {[fill] circle (2pt)
               child {[fill] circle (2pt)}
               child {[fill] circle (2pt)}
             }
       child {[fill] circle (2pt)
               child {[fill] circle (2pt)}
               child {[fill] circle (2pt)}
             }
      };
\node[below=2pt] at (root) {\scriptsize $\delta_3$};
\node[left=4pt] at (root-2) {\scriptsize $[[x_1,x_2],[x_3,x_4]]$};
\node[right=3pt] at (root-1) {\scriptsize $[[x_5,x_6],[x_7,x_8]]$};
\node[left=2pt] at (root-2-2) {\scriptsize $[x_1,x_2]$};
\node[left=3pt] at (root-2-1) {\scriptsize $[x_3,x_4]$};
\node[right=3pt] at (root-1-2) {\scriptsize $[x_5,x_6]$};
\node[right=2pt] at (root-1-1) {\scriptsize $[x_7,x_8]$};
\node[above=2pt] at (root-2-2-2) {\scriptsize $x_1$};
\node[above=2pt] at (root-2-2-1) {\scriptsize $x_2$};
\node[above=2pt] at (root-2-1-2) {\scriptsize $x_3$};
\node[above=2pt] at (root-2-1-1) {\scriptsize $x_4$};
\node[above=2pt] at (root-1-2-2) {\scriptsize $x_5$};
\node[above=2pt] at (root-1-2-1) {\scriptsize $x_6$};
\node[above=2pt] at (root-1-1-2) {\scriptsize $x_7$};
\node[above=2pt] at (root-1-1-1) {\scriptsize $x_8$};
\end{tikzpicture}
\caption{The trees of the words $\gamma_4$ and $\delta_3$.}
\end{figure}
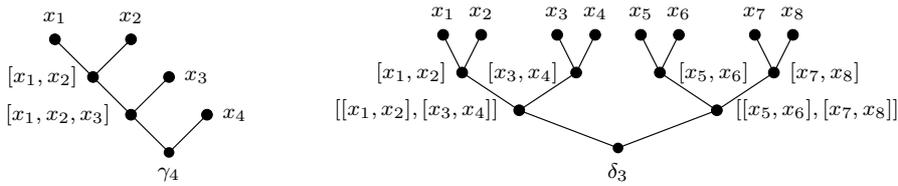
Each of these trees has a visual height, which coincides with the largest
distance from the root to another vertex of the tree.
Observe that the full binary tree of height $i$ corresponds to the derived word
$\delta_i$.
The following two concepts, also introduced in \cite{fernandez-morigi}, will
be essential in our proof of Theorem A.

\begin{dfn}
Let $u$ and $w$ be two outer commutator words.
We say that $u$ is an \emph{extension\/} of $w$ if the tree of $u$ is an
upward extension of the tree of $w$.
If $u\ne w$, we say that $u$ is a \emph{proper extension\/} of $w$.
\end{dfn}

An important remark is that, if $u$ is an extension of $w$, then $G_u\subseteq G_w$.

\begin{dfn}
If $w$ is an outer commutator word whose tree has height $i$, the \emph{defect\/}
of $w$ is the number of vertices that need to be added to the tree of $w$ in
order to get the tree of $\delta_i$.
In other words, if the tree of $w$ has $V$ vertices, the defect of $w$ is
$2^{i+1}-1-V$.
\end{dfn}

Thus the only words of defect $0$ are the derived words.
Our proof of Theorem A also depends on the following result, which is
implicit in the proof of Theorem B of \cite{fernandez-morigi}.

\begin{thm}
\label{product extensions}
Let $w=[u,v]$ be an outer commutator word of height $i$, different from $\delta_i$.
Then at least one of the subgroups $[w(G),u(G)]$ and $[w(G),v(G)]$ is contained
in a product of verbal subgroups corresponding to words which are proper
extensions of $w$ of height $i$.
\end{thm}

Let us now give the proof of Theorem A.

\begin{proof}[Proof of Theorem A]
We argue by induction on the defect of the word $w$.
If the defect is $0$, then $w$ is a derived word, and the result is true
by Theorem \ref{deltai}.
Hence we may assume that the defect is positive.
If the height of $w$ is $i$, let $\Phi=\{\varphi_1,\ldots,\varphi_r\}$ be the set
of all outer commutator words of height $i$ which are a proper extension of $w$.
Since every word in the set $\Phi$ has smaller defect than $w$, the theorem holds for all
$\varphi_i$.

Put $N_0=1$, $N_i=\varphi_1(G)\ldots \varphi_i(G)$ for $1\le i\le r$, and $N=N_r$.
Let us write $w=[u,v]$, where $u$ and $v$ are outer commutator words.
Since $[w(G),w(G)]$ is contained in both $[w(G),u(G)]$ and $[w(G),v(G)]$, it
follows from Theorem \ref{product extensions} that $[w(G),w(G)]\le N$.
Thus if $\overline G=G/N$, the verbal subgroup $w(\overline G)$ is abelian, and
so Theorem A holds in $\overline G$, according to Theorem \ref{nilpotent case}.
Hence $\overline P\cap w(\overline G) = \langle \overline P \cap \overline{G}_{w^m} \rangle$,
and by applying Lemma \ref{from N to L}, we get
$P\cap w(G)=\langle P\cap G_{w^m}, P\cap N \rangle$.

Consequently, it suffices to show that $P\cap N$ can be generated by values
of $w^m$.
We see this by proving that
$P\cap N_i=\langle P\cap N_i\cap G_{w^m} \rangle$ for every $i=0,\ldots,r$,
by induction on $i$.
There is nothing to prove if $i=0$, so we assume that $i\ge 1$.
If $\overline G=G/N_{i-1}$, we have $\overline N_i=\varphi_i(\overline G)$.
Since the theorem is true for $\varphi_i$, it follows that
$\overline P\cap \overline N_i = \langle \overline P \cap \overline{G}_{\varphi_i^m} \rangle$.
By Lemma \ref{from N to L}, we get
\[
P\cap N_i = \langle P\cap G_{\varphi_i^m}, P\cap N_{i-1} \rangle.
\]
Observe that, since $\varphi_i$ is an extension of $w$, we have
$G_{\varphi_i^m}\subseteq G_{w^m}$.
Since also $G_{\varphi_i^m}\subseteq \varphi_i(G)\le N_i$, we can further say that
$G_{\varphi_i^m}\subseteq N_i\cap G_{w^m}$.
Hence
\[
P\cap N_i = \langle P\cap N_i\cap G_{w^m}, P\cap N_{i-1} \rangle,
\]
and the result follows from the induction hypothesis.
\end{proof}

\end{document}